\documentclass{amsart}
\usepackage{amsthm}
\usepackage{amssymb}
\usepackage{mathtools}
\usepackage{graphicx}
\usepackage{enumitem}
\usepackage[all]{xy}
\usepackage{tikz}
\usetikzlibrary{arrows}

\definecolor{lightgrey}{rgb}{.804,.804,.756}
\usepackage[pdfauthor={V. Lebed, A. Mortier},
pdftitle={Quandles and abelianity},
colorlinks=false,linkbordercolor=lightgrey,citebordercolor=lightgrey,urlbordercolor=lightgrey]{hyperref}

\definecolor{myred}{rgb}{.545,0,0}
\definecolor{myblue}{rgb}{.024,.15,.645}    
\definecolor{mygreen}{rgb}{0,.455,0}

\newcommand{\ZZ}{\mathbb{Z}}

\renewcommand{\O}{\mathcal{O}}

\newcommand{\MM}{\mathcal{M}}

\newcommand{\Orb}{\operatorname{Orb}}

\newcommand{\Id}{\operatorname{Id}}
\newcommand{\rk}{\operatorname{rk}}

\newcommand{\Quandle}{\mathbf{Quandle}}
\newcommand{\Grp}{\mathbf{Grp}}

\newcommand\op{\mathrel{\triangleleft}}
\newcommand\wop{\mathrel{\widetilde{\triangleleft}}}

\newcommand{\quotient}[2]{{\left.\raisebox{.2em}{$#1$}\middle/\raisebox{-.2em}{$#2$}\right.}}

\makeatletter
\numberwithin{equation}{section}
\numberwithin{figure}{section}
	\theoremstyle{plain}
\newtheorem{thm}{Theorem}[section]

\newtheorem{cor}[thm]{Corollary}
\newtheorem{pro}[thm]{Proposition}
\newtheorem*{conjecture*}{Conjecture}

\newtheorem{question}[thm]{Question}
	\theoremstyle{definition}
\newtheorem{defn}[thm]{Definition}

	\theoremstyle{remark}
\newtheorem{rem}[thm]{Remark}

\setlist{nolistsep}
\setenumerate{topsep=0pt}
\makeatother

\begin{document}

\title[]{Abelian quandles and \\ quandles with abelian structure group}

\begin{abstract}
Sets with a self-distributive operation (in the sense of $(a \triangleleft b) \triangleleft c = (a \triangleleft c) \triangleleft (b \triangleleft c)$), in particular quandles, appear in knot and braid theories, Hopf algebra classification, the study of the Yang--Baxter equation, and other areas. An important invariant of quandles is their structure group. The structure group of a finite quandle is known to be either ``boring'' (free abelian), or ``interesting'' (non-abelian with torsion). In this paper we explicitly describe all finite quandles with abelian structure group. To achieve this, we show that such quandles are abelian (i.e., satisfy $(a \triangleleft b) \triangleleft c =(a \triangleleft c) \triangleleft b$); present the structure group of any abelian quandle as a central extension of a free abelian group by an explicit finite abelian group; and determine when the latter is trivial. In the second part of the paper, we relate the structure group of any quandle to its 2nd homology group $H_2$. We use this to prove that the $H_2$ of a finite quandle with abelian structure group is torsion-free, but general abelian quandles may exhibit torsion. Torsion in $H_2$ is important for constructing knot invariants and pointed Hopf algebras.
\end{abstract}



\author{Victoria Lebed}
\address{LMNO, Universit\'e de Caen--Normandie, BP 5186, 14032 Caen Cedex, France}
\email{lebed@unicaen.fr}
  
\author{Arnaud Mortier}
\address{LMNO, Universit\'e de Caen--Normandie, BP 5186, 14032 Caen Cedex, France}
\email{arnaud.mortier@unicaen.fr}

\maketitle

\section{Introduction}

A \emph{quandle} is a set $X$ with an idempotent binary operation~$\op$ such that the right translation by any element is a quandle automorphism. In other words, it should satisfy the following axioms for all $a,b,c \in X$:
\begin{enumerate}
\item self-distributivity: $(a \op b) \op c = (a \op c) \op (b \op c)$;
\item the right translation $- \op b$ is a bijection $X \to X$;
\item idempotence: $a \op a = a$.
\end{enumerate}
Removing the last axiom, one gets the notion of \emph{rack}. Groups with the conjugation operation $a \op b = b^{-1}ab$ are fundamental examples of quandles. This yields a functor $\operatorname{Conj} \colon \Grp \to \Quandle$. Numerous other quandle families of various nature are known. The systematic study of self-distributivity was motivated by applications to low-dimentional topology, and goes back to \cite{MR638121, MR672410}.

The \emph{structure group} (also called the \emph{enveloping group}) of a quandle $(X,\op)$ is defined by the following presentation:
\[G{(X,\op)} = \langle g_a, \, a \in X \,|\, g_a g_b = g_b g_{a\op b}, \, a,b \in X \rangle.\]
It brings group-theoretic tools into the study of quandles. More conceptually, it yields a functor $\operatorname{SGr} \colon \Quandle \to \Grp$ which is left adjoint to $\operatorname{Conj}$. The structure group of a rack can be defined along the same lines; however, since the structure groups of a rack and of its associated quandle are isomorphic (see for instance~\cite{LV3}), we treat only the quandle case here.

Structure groups of finite quandles exhibit the following dichotomy:
\begin{enumerate}
\item either they are free abelian of rank $r = \# \Orb (X,\op)$ (the number of orbits of~$X$ with respect to all right translations $- \op b$),
\item or they are non-abelian and have torsion.
\end{enumerate}
In the second case, $G{(X,\op)}$ has a finite index free abelian subgroup of rank~$r$; see \cite{LV3} for more details.

It is natural to ask which quandles fall into the first, ``boring'', category above. The condition $G{(X,\op)} \cong \ZZ$ is easily seen to be equivalent to $X$ being one-element. Quandles with $G{(X,\op)} \cong \ZZ^2$ were completely characterised in \cite{BaNa}. They are parametrised by coprime couples $(m,n)$, with $m\leq n$, and are presented as 
\begin{align}
&U_{m,n} = \{x_0,x_1, \ldots, x_{m-1}, y_0,y_1,\ldots, y_{n-1}\},\label{E:FP2}\\
&x_i \op x_j = x_i, \qquad y_k \op y_l = y_k,\qquad 
x_i \op y_k = x_{i+1}, \qquad  y_k \op x_i = y_{k+1},\notag
\end{align} 
where $0 \leq i,j \leq m-1$, $0 \leq k,l \leq n-1$, and we identify $x_{m}=x_0$ and $y_{n}=y_0$. These quandles were also considered, for different reasons, in~\cite{HomGraphic}.

In this paper we describe all finite quandles with $G{(X,\op)} \cong \ZZ^r$ for arbitrary~$r$ (Theorem~\ref{T:AbIFF}). Up to an action of the symmetric group $S_r$, they are parametrised by $\frac{r^2(r-1)}{2}$ natural numbers subject to some inequalities and a coprimality condition. We simplify this condition in the case $r=3$ (Theorem~\ref{T:Z3IFF}).

To achieve our characterisation, we first show that quandles with abelian structure group are necessarily \emph{abelian}\footnote{Terminology varies a lot in the area: some authors assign the term \emph{abelian} to the property $a \op b = b \op a$, others to $(a \op b) \op (c \op d)=(a \op c) \op (b \op d)$.}, i.e., satisfy the condition
\begin{equation}\label{E:AbelianQ}
(a \op b) \op c = (a \op c) \op b.
\end{equation}
This class of quandles is of independent interest---cf. \cite{2Red2,2Red,Medial,Medial2,HomQ2}. We parametrise $r$-orbit abelian quandles by $\frac{r^2(r-1)}{2}$ natural numbers subject to some inequalities (Theorem~\ref{T:AbQuChar}). This classification is implicit in~\cite{Medial}, where it is derived from structural results on more general {medial} quandles\footnote{I.e., satisfying the condition $(a \op b) \op (c \op d)=(a \op c) \op (b \op d)$. They are also known as \emph{entropic}, and sometimes called~\emph{abelian}.}. Our parametrisation is explicit, which is essential for further results in this paper, and constructive, hence easily programmable. Further, we present the structure group of an abelian $r$-orbit quandle $(X,\op)$ as a central extension of $\ZZ^r$ by an explicit finite abelian group $G'{(X,\op)}$ (Theorem~\ref{T:StrGrpOfAbQu}). Finally, we show that $G'{(X,\op)}$ is trivial (equivalently, $G{(X,\op)}$ is free abelian) if and only if certain greatest common divisor constructed out of the parameters of $(X,\op)$ is trivial.
 
Our result has the following application. The structure group construction extends to set-theoretic solutions $\sigma \colon X \times X \to X \times X$ to the Yang--Baxter equation; the quandle case corresponds to the solutions $(a,b) \mapsto (b, a\op b)$. Structure groups of involutive solutions ($\sigma^2=\Id$) are particularly well understood. One of the tools making involutive solutions accessible is the bijective group $1$-cocycle $G{(X,\sigma)} \to \ZZ^{\# X}$. For a general invertible non-degenerate solution, one has a bijective group $1$-cocycle $G{(X,\sigma)} \to G{(X,\op_\sigma)}$, where $(X,\op_\sigma)$ is the structure rack of $(X,\sigma)$. Thus some results for involutive solutions extend to solutions with $G{(X,\op_\sigma)}$ free abelian. See \cite{GIVdB,ESS,Sol,LYZ,LV,LV3} for more detail.

This discussion raises the following questions:
\begin{question} 
What structural property of a YBE solution corresponds to its structure rack $(X,\op_\sigma)$ being abelian? having abelian structure group $G{(X,\op_\sigma)}$?
\end{question}

The \emph{(rack) homology}\footnote{One could also discuss the \emph{quandle homology} of $(X,\op)$, or consider more complicated coefficients than $\ZZ$. Classical results \cite{LiNe} allow one to reduce these broader contexts to our case.} $H_{\bullet}(X,\op)$ of a quandle $(X,\op)$ is the homology of the following chain complex:
\begin{align}
&C_k (X,\op) = \ZZ X^{k},\notag\\
&d_k(a_1, \ldots, a_{k}) = \sum_{i=2}^{k} (-1)^{i-1} [(a_1,\ldots,\widehat{a_{i}}, \ldots,a_{k})\label{E:d}\\
&\hspace*{3cm} - (a_1 {\op a_i},\ldots,a_{i-1} {\op a_i}, a_{i+1}, \ldots, a_{k})].\notag\
\end{align}
Here $\widehat{a_{i}}$ means that the entry $a_{i}$ is omitted, and the formula for $d_k$ is extended to the whole $\ZZ X^{k}$ by linearity. The rank of $H_{k}(X,\op)$ is known to be $r^k$ (as before, $r = \# \Orb (X,\op)$) \cite{EtGr}. The torsion part of $H_{\bullet}(X,\op)$, which is the part needed for powerful knot invariants and Hopf algebra classification
\cite{RackHom,QuandleHom,AndrGr}, is much less uniform. Even the case of $H_{2}(X,\op)$, the most useful in practice, is understood only for particular families of quandles: Alexander, quasigroup, one-orbit etc. \cite{RackSpace,HomDihedral,Clauwens2,HomTakasaki,HomQuasigroup,GarIglVen,HomAlex}.

In this paper we show that $H_{2}(X,\op)$ is torsion-free for a finite quandle with abelian structure group (Corollary~\ref{C:HomAbStrGrp}). For a general finite abelian quandle, the torsion part of $H_2 (X,\op)$ is a  sum of $r$ (possibly different) quotients of $G'{(X,\op)}$ (Theorem~\ref{T:HomAb}). These quotients can be anything between trivial, like in Corollary~\ref{C:Hom1EltOrbits}, and the whole $G'{(X,\op)}$, like in 
\[H_2(U_{m,n}) \cong \ZZ^4 \oplus G'(U_{m,n})^2 \cong \ZZ^4 \oplus \ZZ^2_{\gcd(m,n)}\]
(Proposition~\ref{P:Hom2orbits}\footnote{This computation appeared before in~\cite{HomGraphic}. Here we recover it using a different method, which we then adapt to several generalisations of $U_{m,n}$. In particular we correct a homology computation from~\cite{HomGraphic}.}). The situation here resembles what happens for one-orbit quandles: there $H_2(X,\op)$ is also controlled by a finite group \cite{GarIglVen}. Our main tool is an explicit group morphism (working for any rack)
\[\prod_{i=1}^r \operatorname{Stab}(a_i,G(X,\op)) \twoheadrightarrow H_2(X,\op),\]
where the $a_i$ are representatives of the orbits of $(X,\op)$, and  the stabiliser subgroups refer to the classical $G(X,\op)$-action on~$X$ (Proposition~\ref{P:GvsH2}).

We finish with an open question:
\begin{question} 
How does the (general degree) homology of a finite abelian quandle depend on its parameters?
\end{question}
In this paper we give examples suggesting that the answer might be rather subtle. In particular we show that the group $G'{(X,\op)}$ does not determine the torsion of $H_{2}(X,\op)$ completely. For instance, the torsion can be trivial without $G'$ being so.


\section{A parametrisation of abelian quandles}\label{S:AbQu}


In this section we classify finite abelian quandles with $r$ orbits. Our description generalises the presentation~\eqref{E:FP2} of the quandles $U_{m,n}$.

Fix a positive integer $r \geq 2$. Take a collection of $\frac{r(r-1)}{2}$ integers
\begin{equation}\label{E:ParMatrix}
M=(m_{i,j})_{1 \leq j \leq i <r}, \qquad \text{ with } 1 \leq m_{i,i}, \text{ and } 0 \leq m_{j,i} < m_{i,i} \text{ for } i < j.
\end{equation}
It can be considered as a lower triangular matrix of size $r-1$. To these parameters we associate an abelian group
\[G(M) = \langle x_1,x_2, \ldots, x_{r-1} \,|\, x_ix_j=x_jx_i, \, x_1^{m_{i,1}}x_2^{m_{i,2}}\cdots x_i^{m_{i,i}} =1 \rangle,\]
where $i$ and $j$ vary between $1$ and $r-1$. In what follows, it will be convenient to use the notations $x_0:=1$ and $m_i:=m_{i,i}$. The group $G(M)$ is finite abelian, of order $m_{1}m_{2}\cdots m_{r-1}$.

For example, for $r=2$ we get a cyclic group of order $m_{1}$, and for $r=4$ and $M=\begin{psmallmatrix}
m_1&&\\ m_{2,1}&m_{2}&&\\m_{3,1}&m_{3,2}&m_{3}
\end{psmallmatrix}$ we get $3$ commuting generators subject to $3$ relations
\begin{align*}
&x_1^{m_{1}}=1,\\
&x_1^{m_{2,1}}x_2^{m_{2}}=1,\\
&x_1^{m_{3,1}}x_2^{m_{3,2}}x_3^{m_{3}}=1.
\end{align*}

Now, take $r$ collections $M^{(1)}$, \ldots, $M^{(r)}$ as above, and consider the disjoint union
\[Q(M^{(1)}, \ldots, M^{(r)})=G(M^{(1)})\sqcup \ldots \sqcup G(M^{(r)}).\]
The generator $x_i$ of $G(M^{(j)})$ will be denoted by $x_i^{(j)}$. We endow $Q(M^{(1)}, \ldots, M^{(r)})$ with a binary operation $\op$ as follows. For any $a^{(i)} \in G(M^{(i)})$ and $b^{(i+k)}\in G(M^{(i+k)})$ (here $0 \leq k < r$, and the sum $i+k$ is considered modulo $r$), put
\[a^{(i)} \op b^{(i+k)} = a^{(i)}x_k^{(i)} \in G(M^{(i)}).\]
In particular, $a^{(i)} \op b^{(i)} = a^{(i)}$. In the simplest case $r=2$, we recover the quandle $U_{m_1^{(1)},m_1^{(2)}}$ from \eqref{E:FP2}. For general $r$, we still get a quandle operation:

\begin{pro}\label{T:FPQu}
The data $(Q(M^{(1)}, \ldots, M^{(r)}),\op)$ above define an abelian quandle. The $r$ components $G(M^{(i)})$ are its orbits. 
\end{pro}

\begin{defn}\label{D:FP}
The quandles above will be called \emph{filtered-permutation}, or \emph{FP}.
\end{defn}

\begin{proof}
Quandle axioms (3) and (2), and the assertion about the orbits, are clear from the construction. Moreover, the groups  $G(M^{(i)})$ are commutative, so all right $\op$-actions commute, hence the abelianity axiom~\eqref{E:AbelianQ}. Let us check the self-distributivity axiom (1). By construction, all elements from the same orbit $G(M^{(i)})$ of $Q:=Q(M^{(1)}, \ldots, M^{(r)})$ right $\op$-act in the same way. Hence for all $a,b,c \in Q$ one has
\[(a \op b) \op c = (a \op c) \op b = (a \op c) \op (b \op c),\]
as required.
\end{proof}

Recall that a quandle $(X,\op)$ is called \emph{$2$-reductive} if the relation
\begin{equation}\label{E:2Red}
a \op (b \op c)= a \op b
\end{equation}
holds for all $a,b,c \in X$.

\begin{thm}\label{T:AbQuChar}
For a finite quandle $(X,\op)$, the following conditions are equivalent:
\begin{enumerate}
\item $(X,\op)$ is abelian;
\item $(X,\op)$ is $2$-reductive;
\item $(X,\op)$ is (isomorphic to) a filtered-permutation quandle.
\end{enumerate}
Moreover, two FP quandles with $r$ \textbf{ordered orbits} are isomorphic if and only if they have the same parameters $M^{(1)}, \ldots, M^{(r)}$.
\end{thm}

\begin{defn}\label{D:pars}
If $(X,\op) \cong Q(M^{(1)}, \ldots, M^{(r)})$, as in (3), we call $M^{(1)}, \ldots, M^{(r)}$ the \emph{parameters} of $(X,\op)$. To make this definition unambiguous, from now on we will work with finite quandles with \textbf{ordered orbits}.
\end{defn}

The equivalence (1) $\Leftrightarrow$ (2) is folklore; the equivalence (1) $\Leftrightarrow$ (3) and the uniqueness assertion are implicit in~\cite{Medial}.

\begin{proof}

(1) $\Rightarrow$ (2). If $(X,\op)$ is abelian, then
\[(a \op b) \op c = (a \op c) \op b = (a \op b) \op (c \op b)  \; \text{ for all } a,b,c \in X.\]
Since the right translation $- \op b$ is bijective, we deduce
\[a \op c = a \op (c \op b) \; \text{ for all } a,b,c \in X.\]

(2) $\Rightarrow$ (1). If $(X,\op)$ is $2$-reductive, then
\[(a \op b) \op c = (a \op c) \op (b \op c) = (a \op c) \op b  \; \text{ for all } a,b,c \in X.\]

(3) $\Rightarrow$ (1) was proved in Proposition~\ref{T:FPQu}.

The implication (1) $\Rightarrow$ (3) requires more work. Let $(X,\op)$ be a finite abelian, hence $2$-reductive, quandle. In particular, $a \op a' = a \op a = a$ for $a$ and $a'$ from the same orbit. Let $O_1,\ldots,O_r$ be the orbits of $(X,\op)$. The $2$-reductivity yields permutations $f_{i,j} \in \operatorname{Perm}(O_i)$, $1 \leq i, j \leq r$ such that
\begin{equation}\label{E:fij}
a \op b =f_{i,j}(a)  \text{ for all } a \in O_i, b \in O_j.
\end{equation}
These permutations satisfy the following conditions:
\begin{enumerate}[label=(\alph*)]
\item commutativity: $f_{i,j}f_{i,k}=f_{i,k}f_{i,j}$;
\item transitivity: the $f_{i,j}$, $1 \leq j \leq r$, generate a transitive subgroup $G_i$ of $\operatorname{Perm}(O_i)$;
\item freeness: $a \cdot g = a$ for some $a \in O_i$ and $g \in G_i$ implies $a' \cdot g = a'$ for all $a' \in O_i$.
\end{enumerate}
Indeed, (a) follows from abelianity, and (b) from the definition of orbits and the finiteness of~$X$. For (c), using transitivity, write $a'=a\cdot h$ for some $h \in G_i$ to get
\[a' \cdot g = (a\cdot h) \cdot g = a\cdot (hg)= a\cdot (gh) = (a\cdot g) \cdot h = a\cdot h = a'.\]

Now, fix an index $i$. All the indices below are considered modulo~$r$. By the freeness, the permutation $f_{i,i+1}$ consists of cycles of the same length; denote this length by $m^{(i)}_1$. Further, take an $a \in O_i$; the permutation $f_{i,i+2}$ will send $a$ to a possibly different $f_{i,i+1}$-cycle, but after $m^{(i)}_2$ iterations will bring it back to the original $f_{i,i+1}$-cycle for the first time. This yields a condition $f_{i,i+2}^{m^{(i)}_2}(a)=f_{i,i+1}^{-m^{(i)}_{2,1}}(a)$ for some $0 \leq m^{(i)}_{2,1} < m^{(i)}_1$. Once again, freeness yields the relation $f_{i,i+1}^{m^{(i)}_{2,1}} f_{i,i+2}^{m^{(i)}_2}=1$ in $\operatorname{Perm}(O_i)$. Similarly, by looking when  $f_{i,i+3}$ brings $a$ back to its original $\left\langle f_{i,i+1},f_{i,i+2}\right\rangle$-orbit (where we are considering the subgroup of $\operatorname{Perm}(O_i)$ generated by $f_{i,i+1}$ and $f_{i,i+2}$), one finds a relation $f_{i,i+1}^{m^{(i)}_{3,1}} f_{i,i+2}^{m^{(i)}_{3,2}} f_{i,i+3}^{m^{(i)}_3}=1$ in $\operatorname{Perm}(O_i)$, with $0 \leq m^{(i)}_{3,1} < m^{(i)}_1$ and $0 \leq m^{(i)}_{3,2} < m^{(i)}_2$. See Fig.~\ref{P:Orbits} for an example: here $r=4$, $i=1$, and $M^{(1)}=\begin{psmallmatrix}
3&&\\ 2&2&&\\0&0&2
\end{psmallmatrix}$.

\definecolor{ttqqqq}{rgb}{0.,0.,0.}
\definecolor{ccqqqq}{rgb}{0.8,0.,0.}
\definecolor{ududff}{rgb}{0.2,0.1,1.}
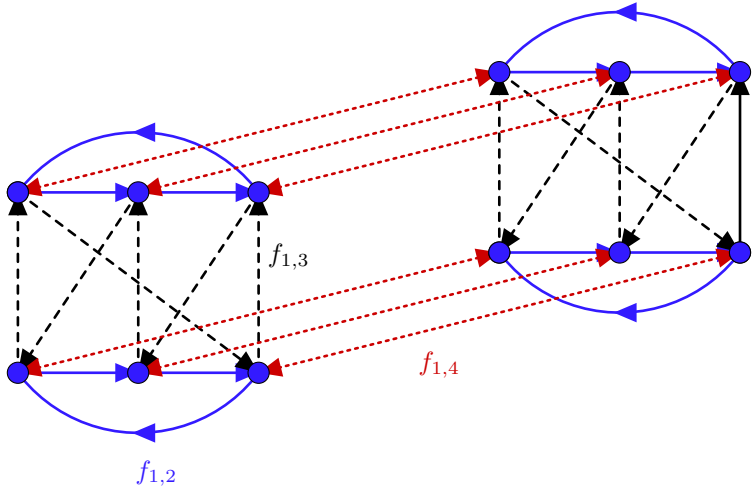
\begin{figure}[h]
\begin{tikzpicture}[line cap=round,line join=round,>=triangle 45,x=.8cm,y=.8cm]
\draw [->,line width=1.pt,color=ududff] (1.,1.) -- (3.,1.);
\draw [->,line width=1.pt,color=ududff] (3.,1.) -- (5.,1.);
\draw [shift={(3.,2.5)},line width=1.pt,color=ududff]  plot[domain=3.7850937623830774:5.639684198386302,variable=\t]({1.*2.5*cos(\t r)+0.*2.5*sin(\t r)},{0.*2.5*cos(\t r)+1.*2.5*sin(\t r)});
\draw [->,line width=1pt,dotted,color=ududff] (3.1,0.) -- (2.9,0.);
\draw [shift={(3.,2.5)},line width=1.pt,color=ududff]  plot[domain=0.6435011087932844:2.498091544796509,variable=\t]({1.*2.5*cos(\t r)+0.*2.5*sin(\t r)},{0.*2.5*cos(\t r)+1.*2.5*sin(\t r)});
\draw [->,line width=1.pt,color=ududff] (1.,4.) -- (3.,4.);
\draw [->,line width=1.pt,color=ududff] (3.,4.) -- (5.,4.);
\draw [->,line width=1.pt,color=ududff] (3.1,5.) -- (2.9,5.);
\draw [->,line width=1.pt,dash pattern=on 3pt off 3pt] (5.,1.) -- (5.,4.);
\draw [->,line width=1.pt,dash pattern=on 3pt off 3pt] (5.,4.) -- (3.,1.);
\draw [->,line width=1.pt,dash pattern=on 3pt off 3pt] (3.,1.) -- (3.,4.);
\draw [->,line width=1.pt,dash pattern=on 3pt off 3pt] (3.,4.) -- (1.,1.);
\draw [->,line width=1.pt,dash pattern=on 3pt off 3pt] (1.,1.) -- (1.,4.);
\draw [->,line width=1.pt,dash pattern=on 3pt off 3pt] (1.,4.) -- (5.,1.);
\draw [->,line width=1.pt,color=ududff] (9.,3.) -- (11.,3.);
\draw [->,line width=1.pt,color=ududff] (11.,3.) -- (13.,3.);
\draw [->,line width=1.pt,color=ududff] (9.,6.) -- (11.,6.);
\draw [->,line width=1.pt,color=ududff] (11.,6.) -- (13.,6.);
\draw [->,line width=1.pt,color=ududff] (11.1,2.) -- (10.9,2.);
\draw [->,line width=1.pt,color=ududff] (11.1,7.) -- (10.9,7.);
\draw [shift={(11.,4.5)},line width=1.pt,color=ududff]  plot[domain=3.7850937623830774:5.639684198386302,variable=\t]({1.*2.5*cos(\t r)+0.*2.5*sin(\t r)},{0.*2.5*cos(\t r)+1.*2.5*sin(\t r)});
\draw [shift={(11.,4.5)},line width=1.pt,color=ududff]  plot[domain=0.6435011087932844:2.498091544796509,variable=\t]({1.*2.5*cos(\t r)+0.*2.5*sin(\t r)},{0.*2.5*cos(\t r)+1.*2.5*sin(\t r)});
\draw [->,line width=1.pt] (13.,3.) -- (13.,6.);
\draw [->,line width=1.pt,dash pattern=on 3pt off 3pt] (13.,6.) -- (11.,3.);
\draw [->,line width=1.pt,dash pattern=on 3pt off 3pt] (11.,3.) -- (11.,6.);
\draw [->,line width=1.pt,dash pattern=on 3pt off 3pt] (11.,6.) -- (9.,3.);
\draw [->,line width=1.pt,dash pattern=on 3pt off 3pt] (9.,3.) -- (9.,6.);
\draw [->,line width=1.pt,dash pattern=on 3pt off 3pt] (9.,6.) -- (13.,3.);
\draw [<->,line width=1.pt,dotted,color=ccqqqq] (5.,1.) -- (13.,3.);
\draw [<->,line width=1.pt,dotted,color=ccqqqq] (3.,1.) -- (11.,3.);
\draw [<->,line width=1.pt,dotted,color=ccqqqq] (1.,1.) -- (9.,3.);
\draw [<->,line width=1.pt,dotted,color=ccqqqq] (5.,4.) -- (13.,6.);
\draw [<->,line width=1.pt,dotted,color=ccqqqq] (3.,4.) -- (11.,6.);
\draw [<->,line width=1.pt,dotted,color=ccqqqq] (1.,4.) -- (9.,6.);
\draw [color=ududff](2.8,-0.3) node[anchor=north west] {$f_{1,2}$};
\draw [color=ttqqqq](5,3.3) node[anchor=north west] {$f_{1,3}$};
\draw [color=ccqqqq](7.5,1.5) node[anchor=north west] {$f_{1,4}$};
\begin{scriptsize}
\draw [fill=ududff] (1.,1.) circle (4pt);
\draw [fill=ududff] (3.,1.) circle (4pt);
\draw [fill=ududff] (5.,1.) circle (4pt);
\draw [fill=ududff] (1.,4.) circle (4pt);
\draw [fill=ududff] (3.,4.) circle (4pt);
\draw [fill=ududff] (5.,4.) circle (4pt);
\draw [fill=ududff] (1.,4.) circle (4pt);
\draw [fill=ududff] (9.,3.) circle (4pt);
\draw [fill=ududff] (11.,3.) circle (4pt);
\draw [fill=ududff] (13.,3.) circle (4pt);
\draw [fill=ududff] (9.,6.) circle (4pt);
\draw [fill=ududff] (11.,6.) circle (4pt);
\draw [fill=ududff] (13.,6.) circle (4pt);
\draw [fill=ududff] (13.,3.) circle (4pt);
\draw [fill=ududff] (9.,3.) circle (4pt);
\end{scriptsize}
\end{tikzpicture}
   \caption{An orbit of a $4$-orbit abelian quandle.}\label{P:Orbits}
\end{figure}

Iterating this argument, one obtains a parameter collection $M^{(i)}$ of the form~\eqref{E:ParMatrix}, and a transitive action of the group $G(M^{(i)})$ on $O_i$: the generator $x^{(i)}_k$ of $G(M^{(i)})$ act by $f_{i,i+k}$. Let us prove that this action is free. If it were not, one would have a relation $f_{i,i+1}^{m'_{k,1}}f_{i,i+2}^{m'_{k,2}}\cdots  f_{i,i+k-1}^{m'_{k,k-1}}f_{i,i+k}^{m'_k}=1$ in $\operatorname{Perm}(O_i)$, with $1 \leq k \leq r-1$, and $0<m'_k<m^{(i)}_k$. But this contradicts the minimality in the choice of $m^{(i)}_k$.

Now, choosing an $a_i \in \operatorname{Perm}(O_i)$ for all $i$, one gets the following identifications:
\begin{align*}
G(M^{(i)}) &\leftrightarrow O_i,\\
(x_1^{(i)})^{n_1}(x_2^{(i)})^{n_2}\ldots (x^{(i)}_{r-1})^{n_{r-1}} &\leftrightarrow f_{i,i+1}^{n_1}f_{i,i+2}^{n_2}\ldots f_{i,i+r-1}^{n_{r-1}}(a_i).
\end{align*}
Moreover, the action of~$f_{i,i+k}$ on~$O_i$ corresponds to multiplying by $x^{(i)}_k$ in $G(M^{(i)})$. One obtains a quandle isomorphism $Q(M^{(1)}, \ldots, M^{(r)}) \cong (X,\op)$, thus (3).

Finally, by the freeness of the $G_i$-action on~$O_i$, the parameter collection $M^{(i)}$ is independent of the choice of the orbit representative~$a_i$, and is thus uniquely determined by the isomorphism class of~$(X,\op)$, where we require isomorphisms to preserve a chosen order of orbits.
\end{proof}

\begin{rem}
The parameters $M^{(1)}, \ldots, M^{(r)}$ describe abelian quandles uniquely up to component reordering, that is, up to the permutation action of the symmetric group~$S_r$. For $r=2$, one gets rid of this redundancy by imposing $m^{(1)}_1\leq m^{(2)}_1$.
\end{rem}

\section{Structure groups of abelian quandles}\label{S:SGofAbQu}

In this section we describe the structure group of a finite abelian quandle in terms of its parameters.

\begin{defn}\label{D:G'}
Let $(X,\op)$ be a finite abelian quandle with parameters $M^{(i)}$. Its \emph{parameter group} is the following quotient of the direct product of the groups $G(M^{(i)})$:
\begin{equation}\label{E:G'}
G'(X,\op):=\quotient{\prod_{i=1}^r G(M^{(i)})}{\left\langle x^{(i)}_{j-i}x^{(j)}_{i-j} \,,\, 1\leq i < j \leq r \right\rangle}.
\end{equation}
\end{defn}

In the simplest case $r=2$ we have
\begin{align*}
G' & \,=\,\quotient{G(M^{(1)})\times G(M^{(2)})}{\left\langle x^{(1)}_{1}x^{(2)}_{1} \right\rangle} \\
& \,\cong\, \left\langle x^{(1)}_{1},x^{(2)}_{1} \,|\, (x^{(1)}_{1})^{m^{(1)}_1}, \, (x^{(2)}_{1})^{m^{(2)}_1}, \, x^{(1)}_{1}x^{(2)}_{1} \right\rangle \,\cong\, \ZZ_{\gcd(m^{(1)}_1,m^{(2)}_1)}.
\end{align*}

\begin{thm}\label{T:StrGrpOfAbQu}
Let $(X,\op)$ be a finite $r$-orbit abelian quandle. Its structure group $G(X,\op)$ is a central extension of $\ZZ^r$ by its parameter group $G'(X,\op)$. Moreover, $G'(X,\op)$ is a finite abelian group, and is (isomorphic to) the commutator subgroup of $G(X,\op)$.
\end{thm}

In the proof we describe this extension explicitly. For $r=2$ it looks as follows:
\[G \cong \left\langle h_1,h_2,q \,|\, h_2h_1=qh_1h_2,\, h_1q=qh_1,\, h_2q=qh_2,\,  q^d=1\right\rangle,\] 
where $d=\gcd(m^{(1)}_1,m^{(2)}_1)$.

In what follows we will often identify $G'$ with the commutator subgroup of $G$.

\begin{proof}
By Theorem~\ref{T:AbQuChar}, it suffices to work with the filtered-permutation quandle $Q:=Q(M^{(1)}, \ldots, M^{(r)})$. The defining relations of its structure group $G$ are
\begin{equation}\label{E:GroupOfFPQuandle}
g_a g_b = g_b g_{ax_{j-i}} \; \text{ for } a \in G(M^{(i)}),\, b \in G(M^{(j)}).
\end{equation} 
As usual, the index $j-i$ is taken modulo $r$. The decorations ${}^{(i)}$ are omitted when clear from the context.

Denote by $G_i$ the subgroup of $G$ generated by the $g_a$ with $a \in G(M^{(i)})$. Since $x^{(i)}_0=1$, it is commutative. Further, one can rewrite \eqref{E:GroupOfFPQuandle} as
\begin{equation}\label{E:GroupOfFPQuandle3}
g_a^{-1} g_b^{-1} g_a g_b = g_a^{-1} g_{ax_{j-i}} \in G_i.
\end{equation}
In particular, this expression is independent of~$b$. Exchanging the roles of $a$ and $b$, one gets
\[g_b^{-1} g_a^{-1} g_b g_a = g_b^{-1} g_{bx_{i-j}} \in G_j,\]
which is independent of~$a$, and is the inverse of the preceding expression. Denoting both sides of~\eqref{E:GroupOfFPQuandle3} by $g_{i,j}$, one thus obtains elements $g_{i,j} \in G_i \cap G_j$ (and hence commuting with $G_i$ and $G_j$), which allow one to break~\eqref{E:GroupOfFPQuandle3} into two parts:
\begin{align}
&g_a g_b = g_{i,j} g_b g_a \; \text{ for } a \in G(M^{(i)}), b\in G(M^{(j)}),\label{E:GroupOfFPQuandle2}\\
&g_{ax_{j-i}} = g_{i,j} g_a \; \text{ for } a \in G(M^{(i)}).\label{E:GroupOfFPQuandle4}
\end{align} 
Moreover, the $g_{i,j}$ satisfy 
\begin{align}
&g_{i,j}g_{j,i}=1 \; \text{ for } 1 \leq i < j \leq r,\label{E:Gij}\\
&g_{i,i}=1 \; \text{ for } 1 \leq i \leq r.\label{E:Gij2}
\end{align}

We will now prove that
\begin{align}
&g_{i,j} \text{ is central in } G \; \text{ for } 1 \leq i, j \leq r. \label{E:Gij3}
\end{align}
Indeed, it can be written as $g_a^{-1} g_{a'}$, with $a$ and $a'$ from the same orbit $G(M^{(i)})$. Taking $b \in G(M^{(j)})$, one computes
\begin{equation}\label{E:Compute}
g_a^{-1} g_{a'}g_b=g_a^{-1} g_{i,j} g_b g_{a'}= g_{i,j} g_a^{-1}  g_b g_{a'}=g_{i,j} g_{j,i} g_b g_a^{-1} g_{a'}=g_b g_a^{-1} g_{a'},
\end{equation}
where we used that $g_{i,j}$ commutes with $G_i$. 

Further, from~\eqref{E:GroupOfFPQuandle4} one sees that the $g_{i,j}$ together with the elements
\[h_i:=g_{1^{(i)}} \in G_i\]
generate the whole group~$G$. Indeed, one can put
\[g_{(x_1^{(i)})^{n_1}(x_2^{(i)})^{n_2}\ldots (x^{(i)}_{r-1})^{n_{r-1}}} = g_{i,i+1}^{n_1}g_{i,i+2}^{n_2}\ldots g_{i,i+r-1}^{n_{r-1}}h_i.\]
This is well defined if and only if one has
\begin{align}
& g_{i,i+1}^{m^{(i)}_{j,1}}g_{i,i+2}^{m^{(i)}_{j,2}}\ldots g_{i,i+j}^{m^{(i)}_{j,j}}=1 \; \text{ for } 1 \leq i \leq r, 1 \leq  j < r. \label{E:Gij4}
\end{align}
If one assumes these conditions, relations~\eqref{E:GroupOfFPQuandle4} become redundant. Finally, since the $g_{i,j}$ are central, it is  sufficient to check relations~\eqref{E:GroupOfFPQuandle2} for the generators $h_i$ only:
\begin{align}
&h_i h_j = g_{i,j} h_j h_i. \label{E:Gij5}
\end{align}

This yields a new presentation for the group $G$:
\begin{equation}\label{E:PresentationG}
G \cong \langle g_{i,j}, \, 1 \leq i,j \leq r \, ; \, h_i, \, 1 \leq i \leq r \;|\; \eqref{E:Gij}, \eqref{E:Gij2}, \eqref{E:Gij3}, \eqref{E:Gij4}, \eqref{E:Gij5} \rangle.
\end{equation}

Next, denote by~$G''$ the subgroup of~$G$ generated by the $g_{i,j}$. From the presentation~\eqref{E:PresentationG}, one sees that $G''$ is the commutator subgroup of~$G$. It is well known\footnote{and true for any rack} that the abelianisation of~$G$ is $\ZZ^r=\oplus_{i=1}^r \ZZ e_i$. Indeed from the relations $g_b g_a = g_a g_b = g_b g_{a \op b}$ in $G_{\mathrm{ab}}$ one deduces that $g_a = g_{a'}$ whenever $a$ and $a'$ lie in the same orbit, thus the map $g_a \mapsto e_i$ for $a$ from the orbit $O_i=G(M^{(i)})$ yields a group isomorphism $G_{\mathrm{ab}} \cong \ZZ^r$. Hence the short exact sequence
\[0 \to G'' \to G \to \ZZ^r \to 0.\]
Since the $g_{i,j}$ are central in~$G$, this presents $G$ as a central extension of $\ZZ^r$ by~$G''$. 

It remains to prove that the groups $G':=G'(X,\op)$ and $G''$ are isomorphic. Relations \eqref{E:Gij}, \eqref{E:Gij2}, and \eqref{E:Gij4} allow one to construct a surjective group morphism 
\begin{align*}
\psi \colon G' &\twoheadrightarrow G'',\\
x^{(i)}_j &\mapsto g_{i,i+j}.
\end{align*} 
To show its injectivity, we will construct a set-theoretic map \[\pi \colon G \to G'\] as follows. Take an element $g \in G$ written using the generators $g_{i,j}$ and $h_i$. Move all the occurrences of $h_1^{\pm 1}$ to the left using the centrality of the $g_{i,j}$ and the twisted commutativity~\eqref{E:Gij5} of the $h_i$. Similarly, move all the occurrences of $h_2^{\pm 1}$ right after the $h_1^{\pm 1}$, and so on. Use the relations $h_i h_i^{-1}= h_i^{-1} h_i = 1$ to get a word of the form $h_1^{k_1}\ldots h_r^{k_r}g''$, where $k_1,\ldots,k_r \in \ZZ$, and $g''$ is a product of the $g_{i,j}{\pm 1}$. Next, in $g''$ replace each generator $g_{i,j}$ by $x^{(i)}_{j-i}$. Denote by $g'$ the word obtained. Considering it as an element of~$G'$, put $\pi(g)=g'$. This is well defined. Indeed, relations \eqref{E:Gij}, \eqref{E:Gij2}, \eqref{E:Gij4}, and $g_{i,j}^{\pm 1}g_{i,j}^{\mp 1}=1$ have counterparts in~$G'$; relation~\eqref{E:Gij3} does not change the result by construction; and neither do~\eqref{E:Gij5} and $h_i^{\pm 1}h_i^{\mp 1}=1$, as shows a computation similar to~\eqref{E:Compute}, combined with~\eqref{E:Gij}. Consider the restriction 
\[\varphi:=\pi|_{G''} \colon G'' \to G'.\]
It simply replaces each $g_{i,j}$ by $x^{(i)}_{j-i}$ in any representative of an element of~$G''$, and is thus the desired inverse of~$\psi$.

Finally, $G'$ is a finite abelian group, since so are the groups $G(M^{(i)})$.
\end{proof}

\section{Quandles with abelian structure group}\label{S:AbIFF}

Finally, we are ready to classify all finite quandles with abelian structure group.

\begin{defn}\label{D:M}
Let $(X,\op)$ be a finite abelian quandle with $r$ orbits. Its \emph{parameter matrix} $\MM(X,\op)$  is constructed from its parameters $M^{(i)}$ as follows. Its columns are indexed by couples $(i,j)$ with $1 \leq i < j \leq r$. Its rows are indexed by couples $(i,j)$ with $1 \leq i \leq r$, $1 \leq j < r$. All couples are ordered lexicographically here. The row $(i,j)$ corresponds to the $j$th row of $M^{(i)}$; for all $1 \le k \le \min \{j, r-i\}$, it contains $m^{(i)}_{j,k}$ in the column $(i,i+k)$, and for all $r-i < k \leq j$, it contains $-m^{(i)}_{j,k}$ in the column $(i+k-r,i)$. 
\end{defn}

For $r=2$ and the parameters $M^{(1)} = \begin{psmallmatrix} m^{(1)}_1 \end{psmallmatrix}$ and $M^{(2)} = \begin{psmallmatrix} m^{(2)}_1 \end{psmallmatrix}$, one gets 
\[\MM(X,\op)=\begin{psmallmatrix}
m^{(1)}_1 \\ -m^{(2)}_1
\end{psmallmatrix}.\]
For $r=3$, there are $3$ columns: $(1,2)$, $(1,3)$, $(2,3)$, and 
\[\MM(X,\op)=\begin{psmallmatrix}
m^{(1)}_1 & \cdot & \cdot\\
m^{(1)}_{2,1} & m^{(1)}_{2} & \cdot\\ 
\cdot & \cdot & m^{(2)}_1\\
-m^{(2)}_{2} & \cdot & m^{(2)}_{2,1}\\
\cdot & -m^{(3)}_1 & \cdot\\
\cdot & -m^{(3)}_{2,1} & -m^{(3)}_{2}\\
\end{psmallmatrix}\]
(the dots are zeroes).

\begin{thm}\label{T:AbIFF}
For a finite quandle $(X,\op)$, the following conditions are equivalent:
\begin{enumerate}
\item the structure group of $(X,\op)$ is abelian;
\item the quandle $(X,\op)$ is abelian, and its parameter group $G'(X,\op)$ is trivial;
\item the quandle $(X,\op)$ is abelian, and the maximal minors of its parameter matrix $\MM(X,\op)$ are globally coprime.
\end{enumerate}
\end{thm}

For $r=2$, the coprimality condition from the theorem becomes $\gcd(m^{(1)}_1,m^{(2)}_1)=1$, and we recover the classification of finite quandles with structure group~$\ZZ^2$ from \cite{BaNa}. For $r=3$, we will simplify the condition from the theorem in Section~\ref{S:Z3IFF}.

\begin{proof}
Let us first show that a finite quandle $(X,\op)$ with abelian structure group $G:=G{(X,\op)}$ is abelian. Indeed, by the construction of the structure group, and due to quandle axioms (1) and (2), the assignment $a \cdot g_b:=a \op b$ extends to a right action of $G$ on~$X$. Since $G$ is abelian, we get
\[(a \op b) \op c = (a\cdot g_b) \cdot g_c = a\cdot (g_b g_c) = a\cdot (g_c g_b) = (a\cdot g_c) \cdot g_b = (a \op c) \op b.\]

Thus we only need to understand which finite abelian quandles have abelian structure group. By Theorem~\ref{T:StrGrpOfAbQu}, this happens if and only the parameter group $G':=G'(X,\op)$ is trivial. Indeed, if $G'$ is trivial, than $G \cong \ZZ^r$; and if $G$ is abelian, then by~\cite{LV3} it is free abelian, hence the only possibility for its finite subgroup~$G'$ is to be trivial. We thus proved (1) $\Longleftrightarrow$ (2).

Let us show the equivalence (2) $\Longleftrightarrow$ (3). Assume the quandle $(X,\op)$ abelian, with parameters $M^{(i)}$. The group~$G'$ admits as generators the elements $x^{(i)}_{j}$ for $i+j \leq r$, since for $i+j > r$ one has $x^{(i)}_{j}=(x^{(j+i-r)}_{r-j})^{-1}$. With these $n:=\frac{r(r-1)}{2}$ generators, $G'$ is isomorphic to the quotient of $\ZZ^n$ by the row space of the matrix $\MM:=\MM(X,\op)$. Indeed, the rows of~$\MM$ encode the defining relations of the components $G(M^{(i)})$ of~$G$, taking into account the identification $x^{(i)}_{j}=(x^{(j+i-r)}_{r-j})^{-1}$. By a classical argument, the triviality of~$G'$ is then equivalent to the maximal minors of $\MM$ being globally coprime. This can be seen as follows: given a finitely generated abelian group, both its isomorphism class and the greatest common divisor of the maximal minors of its presentation matrix as above are invariant under elementary row and column operations, and for a matrix in Smith normal form, the triviality of the group and the minors condition are both equivalent to the matrix being of maximal rank with all diagonal entries equal to $1$.
\end{proof}

One can ask whether there are many quandles satisfying the conditions from the theorem. The answer is yes, as is shown by the following example:

\begin{pro}
Let $(X,\op)$ be a finite abelian quandle with $r$ orbits, and assume that its parameters $m^{(i)}_{j,k}$ vanish whenever $k<j$. Then the following conditions are equivalent:
\begin{enumerate}[label=(\alph*)]
\item the structure group of $(X,\op)$ is $\ZZ^r$;
\item $\gcd(m^{(i)}_{j},m^{(j+i-r)}_{r-j})=1$ whenever $i+j>r$.
\end{enumerate}
\end{pro}

The quandles from the proposition have $r(r-1)$ non-zero parameters $m^{(i)}_{j}$, and condition (b) divides them into $\frac{r(r-1)}{2}$ coprime pairs. One thus obtains, for each~$r,$ an infinite family of quandles with structure group $\ZZ^r$.

\begin{proof}
One could compute the maximal minors from the point (3) of Theorem~\ref{T:AbIFF}. Instead, we choose here to check the triviality of the abelian group $G':=G'(X,\op)$, and use the equivalence (1) $\Longleftrightarrow$ (2) from the theorem. In our situation, $G'$ has the following presentation:
\[G' \,\cong\, \langle \, x^{(i)}_j \,|\, \left( x^{(i)}_j\right)^{m^{(i)}_{j}} =1,\, \, x^{(l)}_{k-l}x^{(k)}_{l-k} =1 \, \rangle,\]
where $1\leq i \leq r$, $1\leq j < r$, and $1\leq l < k \leq r$. The last condition means that the generators $x^{(l)}_{k-l}$ and $x^{(k)}_{l-k}$ are mutually inverse. The above presentation then rewrites as
\[G' \,\cong\, \langle \, x^{(i)}_j, \, i+j>r \,|\, \left( x^{(i)}_j\right)^{\gcd(m^{(i)}_{j},m^{(j+i-r)}_{r-j})} =1, \, i+j>r \, \rangle,\]
where $1\leq i \leq r$, $1\leq j < r$. But this is the direct product of the cyclic groups of orders $\gcd(m^{(i)}_{j},m^{(j+i-r)}_{r-j})$, where $1\leq i \leq r$, $1\leq j < r$, and $ i+j>r$.
\end{proof}

\section{Quandles with structure group $\ZZ^3$}\label{S:Z3IFF}

In the case $r=3$, instead of computing the ${{6}\choose{3}}=20$ maximal minors of the $6 \times 3$ parameter matrix, it is in fact sufficient to compute only $7$ simple greatest common divisors:

\begin{thm}\label{T:Z3IFF}
The structure group of a finite quandle is $\ZZ^3$ if and only if it is abelian with $3$ orbits, and its parameters satisfy the following conditions:
\begin{enumerate}
\item $\gcd(m^{(1)}_1,m^{(1)}_{2,1},m^{(2)}_{2})=\gcd(m^{(1)}_{2},m^{(3)}_1,m^{(3)}_{2,1})=\gcd(m^{(2)}_1,m^{(2)}_{2,1},m^{(3)}_{2})=1$;
\item $\gcd(m^{(1)}_1,m^{(1)}_{2,1}, m^{(2)}_1,m^{(3)}_{2})=\gcd(m^{(1)}_1,m^{(2)}_{2},m^{(3)}_1,m^{(3)}_{2,1})$ 

$=\gcd(m^{(1)}_{2},m^{(2)}_1,m^{(2)}_{2,1},m^{(3)}_1)=1$;
\item $\gcd(m^{(1)}_1,\, m^{(2)}_1,\, m^{(3)}_1,\, m^{(1)}_{2,1}m^{(2)}_{2,1}m^{(3)}_{2,1}-m^{(1)}_{2}m^{(2)}_{2}m^{(3)}_{2})=1$.
\end{enumerate}
\end{thm}

\begin{proof}
By Theorem~\ref{T:AbIFF}, we may assume our quandle abelian with $3$ orbits. For the sake of readability, let us rename the entries of its parameter matrix and permute its rows, to get the matrix
\begin{equation}\label{E:M}
M=\left(\begin{matrix}a&\cdot&\cdot\\
\cdot&b&\cdot\\
\cdot&\cdot&c\\
u&v&\cdot\\ 
w&\cdot&x\\ 
\cdot&y&z\end{matrix}\right).
\end{equation} 
The conditions from the theorem then become:
\begin{enumerate}
 \item $\gcd(a,u,w)=\gcd(b,v,y)=\gcd(c,x,z)=1$;
 \item $\gcd(a,b,w,y)=\gcd(a,c,u,z)=\gcd(b,c,v,x)=1$;
 \item $\gcd(a,b,c,\Delta)=1$,
\end{enumerate}
where $\Delta =uxy+vwz=-\det\left(\begin{matrix}u&v&\cdot\\w&\cdot&x\\\cdot &y&z\end{matrix}\right)$.
By Theorem~\ref{T:AbIFF}, we need to prove that these three conditions are equivalent to the coprimality of the maximal minors of~$M$, which here means
\begin{enumerate}[resume]
 \item $D:= \gcd (abc, abx, abz, avc, avx, avz, ayc, ayx, bcu, bcw, bux, buz, bwz, $
 
  \hspace*{7cm} $cvw, cuy, cwy, \Delta )=1.$
\end{enumerate}
All monomials in these minors contain one element from each column of~$M$, so (4) $\Longrightarrow$ (1). Similarly, all monomials contain, say, one element from the first column and one from the second, and never $u$ and $v$ simultaneously, so $\gcd(a,b,w,y)$ divides $D$. A similar argument for the remaining pairs of columns yields (4) $\Longrightarrow$ (2). Finally, all the minors except for $\Delta$ are divisible by $a$, $b$, or $c$, hence (4) $\Longrightarrow$ (3).

In the opposite direction, $\gcd(a,u,w)=1$ implies $\gcd(abc, bcu, bcw) = bc$. Similarly, $\gcd(b,v,y)=1$ implies $\gcd(abc, avc, ayc) = ac$, and $\gcd(c,x,z)=1$ implies $\gcd(abc, abx, abz) = ab$. Hence (1) allows one to simplify $D$ as
\[D=\gcd(bc, ac, ab, avx, avz, ayx, bux, buz, bwz, cvw, cuy, cwy, \Delta).\]
Also, $\gcd(c,x,z)=1$ implies $\gcd(avc, avx, avz) = av$. Analogous arguments lead to further simplifications:
\[D=\gcd(bc, ac, ab, av, ax, bu, bz, cw, cy, \Delta).\]
Now, $\gcd(a,b,w,y)=1$ yields $\gcd(ac,bc,wc,yc)=c$. Repeating the same argument for other conditions from~(2), one gets
\[D=\gcd(a,b,c, \Delta),\]
which is $1$ by (3).
\end{proof}

One could deduce conditions (1)-(3) above from the triviality of the parameter group $G'$ in a more conceptual way. Indeed, if $G'$ is trivial, it remains so when one forgets any two of its three generators---that is, when one removes any two of the three columns of the matrix~$M$ from~\eqref{E:M}. The maximal minors of the remaining $6 \times 1$ matrices yield conditions (1). Similarly, when one forgets, say, the third generator, one is left with the matrix
\[\left(\begin{matrix}a&\cdot\\
\cdot&b\\
\cdot&\cdot\\
u&v\\ 
w&\cdot\\ 
\cdot&y\end{matrix}\right).\]
Since the relations $x^m=x^n=1$ are equivalent to $x^{\gcd(m,n)}=1$, this matrix defines the same group as the matrix
\[\left(\begin{matrix}\gcd(a,w)&\cdot\\
\cdot&\gcd(b,y)\\
u&v\end{matrix}\right).\]
Applying to this matrix the proposition below, one gets conditions (2).

\begin{pro}
Define the abelian group $G$ as the quotient of $\ZZ^2$ by the row space of the matrix
\[M=\left(\begin{matrix}a&\cdot\\
\cdot&b\\
c&d\end{matrix}\right).\]
Then the following are equivalent:
\begin{enumerate}[label=(\alph*)]
 \item $G$ is trivial;
 \item $\gcd(ab,ad,bc)=1$;
 \item $\gcd(a,b)=\gcd(a,c)=\gcd(b,d)=1$.
\end{enumerate}
\end{pro}

\begin{proof}
\ \begin{description}
 \item[$(a)\Leftrightarrow (b)$] This follows by computing the maximal minors of~$M$. (Cf. the argument at the end of the proof of Theorem~\ref{T:AbIFF}.)
 \item[$(b)\Rightarrow (c)$] This follows from the obvious inclusion of $\left< ab,ad,bc \right>$ in the three subgroups of $\ZZ$ that are $\left< a,b \right>$, $\left< a,c \right>$, $\left< b,d \right>$.
\item[$(c)\Rightarrow (b)$] Let $u,v\in\ZZ$ be such that $au+bv=1$. Let $p,q,r,s\in\ZZ$ be such that $pb+qd=u$ and $ra+sc=v$. Then \[(ab)(p+r)+(ad)q+(bc)s =1\] implies $(b)$ as desired. \qedhere
\end{description}
\end{proof}

Finally, assume $p:=\gcd(a,b,c,\Delta)$ greater than $1$. Our group $G'$ remains trivial when one requires the $p$th powers of its generators to vanish. This corresponds to considering the coefficients of the matrix~$M$ from~\eqref{E:M} modulo $p$. Since $p$ divides $a$, $b$, and $c$, the first three rows of the matrix obtained vanish. Since $p$ also divides $\Delta=-\det\left(\begin{matrix}u&v&\cdot\\w&\cdot&x\\\cdot &y&z\end{matrix}\right)$, all maximal minors of our matrix vanish. But for the group to be trivial, these maximal minors have to be coprime.

\section{Structure group vs homology: path maps}\label{S:Hom}

Before investigating the homology of abelian quandles, let us describe a relation between the structure group $G$ and the second homology group $H_2$ of any quandle (or even rack) $(X,\op)$. To do this, we will reverse the usual order in the definition of $H_2$ (restrict to the kernel $Z_2$ of $d_2$, then mod out the image $B_2$ of $d_3$): we will instead consider the quotient $Q_2:=C_2/B_2$ before restricting it to $H_2=Z_2/B_2$. 

We will use the classical rack homology decomposition. Let $O_i$ be the orbits of $(X,\op)$. By the formula~\eqref{E:d}, the differentials $d_k$ preserve the decomposition 
\begin{equation}\label{E:OrbitDecomposition}
C_k (X,\op) = \bigoplus_{i} \ZZ O_i \times X^{k-1}.
\end{equation}
For $L \in \{C,Z,B,Q,H\}$, denote by $L_{k;i}$ the part of $L_k$ corresponding to $\ZZ O_i \times X^{k-1}$.

Recall also  the classical (truncated) topological realisation $BX$ for $Q_2$ \cite{RackSpace}: it consists of $X$-labelled vertices, $X$-labelled directed edges $a \, \overset{b}{\to} \, a \op b$ (corresponding to the generator $(a,b)$ of $C_2$), and squares of the form 
\[\xymatrixcolsep{4pc}\xymatrixrowsep{1pc}\xymatrix{
a \op c \ar[r]^-{b \op c} & (a \op b) \op c \\
a \ar[r]^b \ar[u]^c         & a \op b \ar[u]^c }\]
The homology group $H_2$ of a quandle is the $1$st homology group $H_1(BX,\ZZ)$ of this space, since the boundary of the edge $a \, \overset{b}{\to} \, a \op b$ coincides with
\[d_2(a,b)= a \op b - a,\]
and the boundary of a square as above coincides with 
\[d_3(a,b,c) = (a \op b,c) - (a,c) - (a \op c, b \op c) + (a,b).\]
The orbit decomposition~\eqref{E:OrbitDecomposition} becomes the connected component decomposition for the CW space $BX$.

Now, fix an $a \in X$, and take a $g \in G$ written as a word $w$ in the generators $g_b$. Consider the path in $BX$ starting from the vertex $a$ and consisting of edges labelled by the letters from $w$; an edge points to the right or to the left depending on whether the corresponding generator or its inverse is in $w$. The labels of the remaining vertices are reconstructed from the edge labels in a unique way. The rightmost vertex label will be denoted by $a \cdot g$; it will be shown to be independent of the choice of the representative $w$ of $g$, and to yield the classical $G$-action on~$X$.

Here is an example with $w=g_b g_c^{-1} g_d$: 
\[\xymatrixcolsep{3pc}\xymatrix{
a \ar[r]^-b & a \op b & (a \op b) \wop c \ar[l]_-c\ar[r]^-d & ((a \op b) \wop c) \op d}.\]
As usual, $- \wop c$ is the inverse of the right translation $- \op c$. This path corresponds to (the class of) the element
\[(a,b)- ((a \op b) \wop c,c)+((a \op b) \wop c,d) \,\in\, Q_2,\]
and we have $a \cdot g_b g_c^{-1} g_d =((a \op b) \wop c) \op d$.

\begin{pro}\label{P:GvsH2}
Let $(X,\op)$ be a rack. Fix an $a \in X$ lying in the orbit $O_i$. The construction above defines a (set-theoretic) map
\[p_a \colon G(X,\op) \to Q_{2;i}(X,\op).\]
It restricts to a surjective group morphism
\[p'_a \colon G_a \twoheadrightarrow H_{2;i}(X,\op),\]
where $G_a$ is the stabiliser subgroup of~$a$ in $G(X,\op)$ for the action~$\cdot$ above.
\end{pro}

The maps $p'_a$ will help us deduce things about the cohomology of abelian quandles from what we know about their structure groups. We hope that in other situations they might also transport insights about homology to structure groups.

\begin{proof}
We need to check that $p_a$ is compatible with the relations $g_b g_b^{-1} = g_b^{-1} g_b = 1$ and $g_b g_c g_{c \op b}^{-1} g_c^{-1} =1$ in $G:=G(X,\op)$. By construction, consecutive $g_b$ and $g_b^{-1}$ are sent to the same edge travelled in opposite directions, which can be omitted. The expression $g_b g_c g_{c \op b}^{-1} g_c^{-1}$ is sent to the boundary of a square, hence can be omitted in~$Q_2$ as well. Thus the map $p_a$ is well defined. In particular, the rightmost vertex label of $p_a(g)$, denoted by $a \cdot g$, is well defined, and yields a transitive right action of~$G$ on $O_i$. This action is determined by the property $a \cdot g_b = a \op b$ for all $a,b \in X$.

By construction, we have
\begin{align}
&p_a(gg') = p_a(g)p_{a \cdot g} (g'),\label{E:1cocycle}\\
&d_2(p_a(g)) = a\cdot g - a\notag
\end{align}
for all $g,g' \in G$. If $g$ fixes $a$, these become $p_a(gg')=p_a(g)p_{a} (g')$ and $d_2(p_a(g))=0$, so $p_a$ restricts to a group morphism $G_a \to H_{2;i}$. 

It remains to check that this restriction $p'_a$ is surjective. Elements of~$H_{2;i}$ are linear combinations of classes of loops in $BX$. If a loop representative starts at some $a' \in O_i$, we may conjugate it by a path connecting $a$ to $a'$, as $a$ and $a'$ lie in the same orbit~$O_i$. This does not change the homology class of the loop. Hence each loop is in the image of $p'_a$. 
\end{proof}

\begin{defn}\label{D:pa}
The maps $p_a$ above will be referred to as \emph{path maps}.
\end{defn}

By~\eqref{E:1cocycle}, path maps are group $1$-cocycles. 

\begin{rem}\label{R:IndepOfa}
For $a$ and $a'$ from the same orbit, the stabiliser subgroups $G_a$ and $G_{a'}$ are  related by a conjugation in~$G$, which intertwines the restricted path maps $p'_a$ and $p'_{a'}$.
\end{rem}

\section{Quandles with abelian structure group have torsion-free $H_2$}\label{S:HomAb}

For abelian quandles, path maps relate the torsion of~$H_2$, which is the interesting part for applications, to the parameter group $G'$, which we studied above:

\begin{thm}\label{T:HomAb}
Let $(X,\op)$ be a finite abelian quandle with $r$ orbits. Then 
\begin{equation}\label{E:HomAbQu}
H_2(X,\op) \cong \ZZ^{r^2} \bigoplus \oplus_{i=1}^r T_i,
\end{equation}
where the finite groups $T_i$ are all quotients of the parameter group $G'(X,\op)$.
\end{thm}

More precisely, $T_i$ is the image of~$G'$ (seen as the commutator subgroup of $G$) by the path map $p_a$ for any $a$ from the orbit~$O_i$.

\begin{defn}
The groups $T_i$ will be called the \emph{torsion groups} of $(X,\op)$.
\end{defn}

By Theorem~\ref{T:AbIFF}, the parameter group of a finite quandle with abelian structure group is trivial. Hence all its quotients $T_i$ are trivial as well, and we obtain

\begin{cor}\label{C:HomAbStrGrp}
Let $(X,\op)$ be a finite $r$-orbit quandle with abelian structure group. Then its 2nd homology group is torsion-free: $H_2(X,\op) \cong \ZZ^{r^2}$.
\end{cor}

The converse of this corollary is false: in Proposition~\ref{P:graphic} we will describe abelian quandles with non-abelian structure group and torsion-free~$H_2$.

\begin{proof}[Proof of Theorem \ref{T:HomAb}]
We will show the decomposition $H_{2;i} \cong \ZZ^{r} \oplus T_i$ for all $1 \leq i \leq r$, which implies~\eqref{E:HomAbQu}.

Let $O_1,\ldots,O_r$ be the orbits of $(X,\op)$. Fix an $a \in O_i$, and recall the restricted path map $p'_a \colon G_a \twoheadrightarrow H_{2;i}$. 

Since our quandle is abelian, commutators in~$G$ act trivially on any element of~$X$, so the commutator subgroup $G'$ is a subgroup of the stabiliser subgroup~$G_a$. The subgroup $G'$ is normal (even central) in~$G$, hence in~$G_a$. So, $p'_a$ induces a surjective group morphism $\overline{p}'_a \colon G_a/G' \twoheadrightarrow H_{2;i}/p'_a(G')$. The group $T_i:=p'_a(G')$  is an isomorphic image, hence a quotient, of $G'$. It is finite since $G'$ is so. By Remark~\ref{R:IndepOfa}, it is independent of the choice of the representative $a$ of the orbit~$O_i$.

Further, the inclusion $G_a \hookrightarrow G$ induces an inclusion $G_a/G' \hookrightarrow G/G'$. One can assemble everything in a commutative diagram, where all arrows but $p_a$ are group morphisms, and the three squares commute:
\[\xymatrixcolsep{3pc}\xymatrixrowsep{1pc}\xymatrix{
G \ar@{->>}[d] \ar[rr]^-{p_a} & & Q_{2;i} &\\
G/G'& G_a \ar@{_{(}->}[ul] \ar@{->>}[rr]^-{p'_a} \ar@{->>}[d] && H_{2;i} \ar@{->>}[d]  \ar@{_{(}->}[ul] \\
& G_a/G' \ar@{_{(}->}[ul] \ar@{->>}[rr]^-{\overline{p}'_a} && H_{2;i}/T_i
}\]
Here all the maps $\twoheadrightarrow$ and $\hookrightarrow$ are the obvious quotients and inclusions. In what follows they will all be abusively denoted by $\pi$ and $\iota$ respectively.

We will now prove that $\overline{p}'_a$ is injective, hence a group isomorphism. This will give the short exact sequence
\[0 \to T_i \to H_{2;i} \to G_a/G' \to 0.\]
The group $G_a/G'$ is a subgroup of $G/G' = G_{\mathrm{ab}} \cong \ZZ^r$ (cf. the proof of Theorem~\ref{T:StrGrpOfAbQu}). As a result, $G_a/G' \cong \ZZ^{r'}$ for some $r' \leq r$. Our short exact sequence becomes
\[0 \to T_i \to H_{2;i} \to \ZZ^{r'} \to 0.\]
Since $\ZZ^{r'}$ is free abelian, the sequence splits: $H_{2;i} \cong \ZZ^{r'} \oplus T_i$. The group $T_i$ being finite, we have $\rk(H_{2;i})=r'$. From \cite{EtGr} (or from a direct inspection of the orbit $O_i=G(M^{(i)})$), we get $\rk(H_{2;i})=r$, hence $r=r'$, and $H_{2;i} \cong \ZZ^{r} \oplus T_i$ as desired.

To prove the injectivity of~$\overline{p}'_a$, we need the group isomorphism $\overline{\pi}_G \colon G/G'= G_{\mathrm{ab}} \overset{\sim}{\to} \ZZ^r =\oplus_{j=1}^r \ZZ e_j$ induced by the map $\pi_G \colon G \twoheadrightarrow \ZZ^r$ sending $g_b$ to $e_j$ for all $b$ from the orbit $O_j$ (cf. the proof of Theorem~\ref{T:StrGrpOfAbQu}). Similarly, the assignment $(a',b) \mapsto e_j$ for $b \in O_j$ induces a map $\pi_Q \colon Q_{2;i} \twoheadrightarrow \ZZ^r$; indeed, $\pi_Q$ sends boundaries $d_3(a',b,c) = (a' \op b,c) - (a',c) - (a' \op c, b \op c) + (a',b)$ to $0$. Recalling the definition of the path map $p_a$, one sees that it intertwines $\pi_G$ and $\pi_Q$: $\pi_G = \pi_Q p_a$. Finally, the group $T_i=p'_a(G')$ is generated by (the classes of) the loops of the form
\[\xymatrixcolsep{4pc}\xymatrixrowsep{1pc}\xymatrix{
a \op c \ar[r]^-{b} & (a \op b) \op c \\
a \ar[r]^b \ar[u]^c         & a \op b \ar[u]^c }\]
The map $\pi_Q$ sends them to~$0$, and thus induces a map $\pi_T \colon H_{2;i}/T_i \to \ZZ^r$. One can now extend the above commutative diagram by two triangles and one square, all of which commute:
\[\xymatrixcolsep{2.5pc}\xymatrixrowsep{1.5pc}\xymatrix{
&& \ZZ^r &&&\\
&G \ar@{->>}[d] \ar[rr]^-{p_a} \ar@{->>}[ur]^-{\pi_G} & & Q_{2;i} \ar@{->>}[ul]_-{\pi_Q}  &&\\
&G/G' \ar `l[ul] `[uur]_-{\simeq}^-{\overline{\pi}_G} [uur] & G_a \ar@{_{(}->}[ul] \ar@{->>}[rr]^-{p'_a} \ar@{->>}[d] && H_{2;i} \ar@{->>}[d]  \ar@{_{(}->}[ul] &\\
&& G_a/G'\ar@{_{(}->}[ul] \ar@{->>}[rr]^-{\overline{p}'_a} && H_{2;i}/T_i \ar `r[ur] `[uuul]^-{\pi_T} [uuull] &
}\]
From this diagram, one reads
\[\pi_T \overline{p}'_a \pi = \pi_T \pi p'_a = \pi_Q \iota p'_a = \pi_Q p_a \iota = \pi_G \iota = \overline{\pi}_G \pi \iota = \overline{\pi}_G \iota \pi.\]
Since $\pi$ is surjective, this implies $\pi_T \overline{p}'_a = \overline{\pi}_G \iota$. Both $\overline{\pi}_G$ and $\iota$ being injective, so is $\overline{p}'_a$, as desired.
\end{proof}

\section{Structure group vs homology: examples}\label{S:HomAbEx}

Let us now see how the parameter group $G'$ and the torsion groups $T_i$ look like in particular cases.

We will start with the rank~$2$ case, i.e. with the quandles $U_{m,n}$. In Section~\ref{S:SGofAbQu} we determined  their parameter groups:
\[G'(U_{m,n}) \cong \ZZ_{\gcd(m,n)}.\]

\begin{pro}\label{P:Hom2orbits}
The 2nd homology group of a $2$-orbit abelian quandle $U_{m,n}$ is
\[H_2(U_{m,n}) \cong \ZZ^4 \oplus \ZZ^2_{\gcd(m,n)}.\]
In particular, its torsion groups both coincide with the whole parameter group:
\[T_1 \cong T_2 \cong G'.\]
\end{pro}

\begin{proof}
Put $d=\gcd(m,n)$, $O_1=\{x_0,x_1, \ldots, x_{m-1}\}$, $O_2=\{y_0,y_1,\ldots, y_{n-1}\}$. We will construct a map $\varphi \colon Q_{2;1} \to \ZZ_d$ sending (the class of) $(x_0,x_1)-(x_0,x_0)$ to $1$. Here we used the description~\eqref{E:FP2} of $U_{m,n}$. This shows that the order of $(x_0,x_1)-(x_0,x_0)$ is at least $d$. In the proof of Theorem~\ref{T:AbQuChar} we saw that $g_{x_0}^{-1}g_{x_0 \op y_0} = g_{x_0}^{-1}g_{x_1}$ generates the parameter group $G' \cong \ZZ_d$, therefore $\overline{p}'_{x_0}(g_{x_0}^{-1}g_{x_1})=-(x_0,x_0)+(x_0,x_1)$ generates~$T_1$. Hence $T_1 \cong \ZZ_d$. Similarly,  $T_2 \cong \ZZ_d$. Theorem~\ref{T:HomAb} allows us to conclude.
 
To describe the map $\varphi$, we need the map  
\begin{align*}
\overline{\bullet} \colon U_{m,n} &\to \ZZ_d,\\
x_i &\mapsto i \mod d,\\
y_k &\mapsto -k \mod d.
\end{align*}
We have
\[\overline{a \op b}=\begin{cases} \overline{a}+1 &\text{ if } a \in O_1, b \in O_2,\\
\overline{a}-1 &\text{ if } a \in O_2, b \in O_1,\\
\overline{a} &\text{ otherwise.}
\end{cases}\]
Further, extend the assignment
\begin{align*}
(a,b) &\mapsto \overline{b}-\overline{a}
\end{align*}
to a map $\psi \colon C_{2,1} \to \ZZ_d$ by linearisation. Let us check that it induces a map $\varphi \colon Q_{2;1} \to \ZZ_d$. We have
\begin{align*}
\psi(d_3(a,b,c)) &= \psi((a \op b,c) - (a,c) - (a \op c, b \op c) + (a,b))\\
&=\overline{c}-\overline{a \op b} -\overline{c}+\overline{a} - \overline{b \op c}+\overline{a \op c}+\overline{b}-\overline{a}\\
&=-\overline{a \op b} - \overline{b \op c}+\overline{a \op c}+\overline{b}.
\end{align*}
Testing all possibilities for the orbits of $a$, $b$ and $c$, one sees that $\psi(d_3(a,b,c))$ always vanishes, so $\psi$ indeed survives in the quotient $Q_{2;1}$. Further, as announced,
\[\psi((x_0,x_1)-(x_0,x_0))=(1-0)-(0-0)=1.\qedhere\] 
\end{proof}

To construct an example where not all the $T_i$ are the same, we need

\begin{pro}\label{P:HomSmallOrbits}
Given a finite abelian quandle, the order of any element in the torsion group $T_i$ divides the square of the size of the orbit~$O_i$.
\end{pro}

\begin{proof}
Fix an $a' \in O_i$, and put $n:=\# O_i$. As seen in the proof of Theorem~\ref{T:AbQuChar}, the elements $g_{b}^{-1}g_{b \op c}$ generate $G'$, therefore the elements $\overline{p}'_{a'}(g_{b}^{-1}g_{b \op c})=-(a,b)+(a, b \op c)$, where $a= a' \wop b$, generate $T_i$. Thus it suffices to show that $n^2(-(a,b)+(a,b \op c))=0$ in $Q_{2;i}$ for all $a \in O_i$, $b,c \in X$. Again by the proof of Theorem~\ref{T:AbQuChar}, the element $g_{b}^{-1}g_{b \op c}$ of~$G'$ depends only on the orbits of $b$ and $c$. This yields
\[(g_{b}^{-1}g_{b \op c})^n=(g_{b}^{-1}g_{b \op c})(g_{b \op c}^{-1}g_{(b \op c) \op c}) \cdots (g_{b\op c^{n-1}}^{-1}g_{b \op c^n})=g_{b}^{-1}g_{b \op c^n},\]
where $b \op c^k$ stands for $(\cdots((b \op c)\op c)\cdots)\op c$, with $k$ occurrences of~$c$. Since $g_{b}^{-1}g_{b \op c}$ is central in~$G$, so is $g_{b}^{-1}g_{b \op c^n}$, and we get
\[(g_{b}^{-1}g_{b \op c})^{n^2}=(g_{b}^{-1}g_{b \op c^n})^n=g_{b}^{-n}g_{b \op c^n}^n.\]
Further, the defining relations of the structure group yield
\[g_{b \op c^n} = g_c^{-1} g_{b \op c^{n-1}}g_c = \ldots = g_c^{-n} g_b g_c^n,\]
so $(g_{b}^{-1}g_{b \op c})^{n^2} = g_{b}^{-n} g_c^{-n} g_b^n g_c^n$. As shown in the proof of Theorem~\ref{T:AbQuChar}, the right translation $- \op b$ divides $O_i$ into cycles of equal length, which has to divide $n=\# O_i$. Hence $g_b^n$ stabilises $a \in O_i$: $g_b^n \in G_a$. The same holds for $g_c^n$. Then
\begin{align*}
n^2(-(a,b)&+(a,b \op c)) = \overline{p}'_{a'}((g_{b}^{-1}g_{b \op c})^{n^2}) = \overline{p}'_{a'}(g_{b}^{-n} g_c^{-n} g_b^n g_c^n) \\
&= -\overline{p}'_{a'}(g_{b}^{n})- \overline{p}'_{a'}(g_c^{n})+ \overline{p}'_{a'}(g_b^n)+ \overline{p}'_{a'}(g_c^n) = 0,
\end{align*}
as desired.
\end{proof}

\begin{rem}
Along the same lines, one shows that $a \op b^n = a \op c^m = a$ for some (equivalently, any) $a \in O_i$, $b \in O_j$, $c \in O_k$ implies that the order of the generator $\overline{p}'_{a}(x^{(j)}_{k-j})$ in $T_i$ divides $mn$. The optimal choice for $m$ and $n$ is the order of $x_{j-i}^{(i)}$ and $x_{k-i}^{(i)}$ in $G(M^{(i)})$ respectively.
\end{rem}

Proposition~\ref{P:HomSmallOrbits} directly implies
\begin{cor}\label{C:Hom1EltOrbits}
Given a finite abelian quandle with a $1$-element orbit $O_i$, its torsion group $T_i$ is trivial.
\end{cor}

To get a concrete example, let us extend the quandle $U_{m,n}$ by adding an element~$z$, and putting
\begin{equation}\label{E:QuExt}
a \op b = a \qquad \text{ whenever } a=z \text{ or } b=z.
\end{equation}
One gets a $3$-orbit abelian quandle, denoted by $U^*_{m,n}$. Its parameter matrix is
\[\begin{psmallmatrix}
m & \cdot & \cdot\\
\cdot & 1 & \cdot\\ 
\cdot & \cdot & 1\\
-n & \cdot & \cdot\\
\cdot & -1 & \cdot\\
\cdot & \cdot & -1\\
\end{psmallmatrix}.\]

\begin{pro}
The parameter group of the quandle $U^*_{m,n}$ is 
\[G'(U^*_{m,n}) \cong \ZZ_{\gcd(m,n)}.\] 
Its 2nd homology group is
\[H_2(U^*_{m,n}) \cong \ZZ^9 \oplus \ZZ^2_{\gcd(m,n)}.\]
In particular, its torsion groups are
\[T_1 \cong T_2 \cong G', \qquad T_3 \cong \{0\}.\]
\end{pro}

Observe that $U_{m,n}$ and $U^*_{m,n}$ have the same torsion in~$H_2$.

\begin{proof}
The parameter group is easily computed from the parameter matrix (recall that the columns represent generators, and the rows relations).

Since the orbit $\O_3$ of $U^*_{m,n}$ is one-element, by Corollary~\ref{C:Hom1EltOrbits} the torsion group~$T_3$ is trivial. To get $T_1\cong T_2 \cong \ZZ_{\gcd(m,n)}$, one can extend the map~$\psi$ from the proof of Proposition~\ref{P:Hom2orbits} from $U_{m,n}$ to $U^*_{m,n}$ by putting $\psi(a,b)=0$ whenever $a=z$ or $b=z$. 
\end{proof}

We continue with computations for one more family of abelian quandles. In particular we obtain two $3$-orbit quandles ($U^*_{2,2}$ and $U^{\star}_{2,2}$) having the same parameter group~$G'$ but different homology groups~$H_2$.

Extend the quandle $U_{m,n}$ by two elements $z_0$ and $z_1$, with
\begin{align*}
a \op z_s &= a \qquad \text{ for all } a,\\
z_s \op x_i &= z_s \op y_k = z_{s+1}.
\end{align*}
Here and below $s \in \{0,1\}$, and the index $s+1$ is taken modulo $2$. One gets a $3$-orbit abelian quandle, denoted by $U^{\star}_{m,n}$. Its parameter matrix is
\[\begin{psmallmatrix}
m & \cdot & \cdot\\
\cdot & 1 & \cdot\\ 
\cdot & \cdot & 1\\
-n & \cdot & \cdot\\
\cdot & -2 & \cdot\\
\cdot & -1 & -1\\
\end{psmallmatrix}.\]

\begin{pro}\label{P:star}
The parameter group of the quandle $U^{\star}_{m,n}$ is 
\[G'(U^{\star}_{m,n}) \cong \ZZ_{\gcd(m,n)}.\] 
Its 2nd homology group is
\[H_2(U^{\star}_{m,n}) \cong \ZZ^9 \oplus \ZZ^2_{\gcd(m,n)} \oplus \ZZ_{\gcd(m,n,2)}.\]
In particular, its torsion groups are
\[T_1 \cong T_2 \cong G', \qquad T_3 \cong \ZZ_{\gcd(m,n,2)},\]
so that $T_3 \cong G'$ if and only if $\gcd(m,n) \in \{1,2\}$.
\end{pro}

\begin{proof}
As usual, the parameter group can be computed from the parameter matrix.

For $T_1$ and $T_2$, the proof we gave for $U^*_{m,n}$  repeats verbatim. For $T_3$, we need to compute the order $t$ of $\overline{p}'_{z_0}(g_{x_0}^{-1}g_{x_0 \op y_0})$ in $H_2$. We know that $t$ divides the order of $g_{x_0}^{-1}g_{x_0 \op y_0} = g_{x_0}^{-1}g_{x_1}$ in~$G'$, which is $\gcd(m,n)$. On the other hand, in $G$ we have
\[g_yg_xg_yg_y = g_yg_y g_{(y \op y) \op y} g_{(x \op y) \op y} =  g_yg_y g_yg_{(x \op y) \op y},\]
where $x:=x_0$ and $y:=y_0$. Since $g_yg_x$, $g_yg_y$ and $g_y g_{(x \op y) \op y}$ stabilise $z_0$, this yields
\[\overline{p}'_{z_0}(g_yg_x) + \overline{p}'_{z_0}(g_yg_y) = \overline{p}'_{z_0}(g_yg_y) + \overline{p}'_{z_0}(g_y g_{(x \op y) \op y}),\]
hence $\overline{p}'_{z_0}(g_yg_x) = \overline{p}'_{z_0}(g_y g_{(x \op y) \op y})$, and
\begin{align*}
0&=\overline{p}'_{z_0}((g_yg_x)^{-1}(g_yg_{(x \op y) \op y}))=
\overline{p}'_{z_0}(g_x^{-1}g_{(x \op y) \op y})\\
&=\overline{p}'_{z_0}((g_x^{-1}g_{x \op y})(g_{x \op y}^{-1}g_{(x \op y) \op y}))=\overline{p}'_{z_0}(g_x^{-1}g_{x \op y})+\overline{p}'_{z_0}(g_{x \op y}^{-1}g_{(x \op y) \op y})\\
&=2\overline{p}'_{z_0}(g_x^{-1}g_{x \op y}).
\end{align*} 
So, $t$ divides $2$, and hence $\gcd(m,n,2)$. It remains to show that, if $m$ and $n$ are both even, then $t=2$. For this we will construct a map $\theta \colon Q_{2;3} \to \ZZ_2$ not vanishing on (the class of) $\overline{p}'_{z_0}(g_x^{-1}g_{x \op y}) = (z_1,x_1)-(z_1,x_0)$. Put
\begin{align*}
\overline{\bullet} \colon U^{\star}_{m,n} &\to \ZZ_2,\\
x_i &\mapsto i \mod 2,\\
y_k &\mapsto k \mod 2,\\
z_s &\mapsto 0.
\end{align*}
For $a \in O_i$ and  $b \in O_j$, we have
\[\overline{a \op b}=\begin{cases} \overline{a}+1 &\text{ if } \{i,j\}=\{1,2\},\\
\overline{a} &\text{ otherwise.}
\end{cases}\]
Further, put 
\begin{align*}
\varepsilon(a,b) &= \begin{cases} 1 &\text{ if } a=z_0, b \in O_1,\\
0 &\text{ otherwise},
\end{cases}
\end{align*}
and extend the assignment $(a,b) \mapsto \varepsilon(a,b)+\overline{b}$ to a map $\psi \colon C_{2,3} \to \ZZ_2$ by linearisation. Let us check that it induces a map $\theta \colon Q_{2;3} \to \ZZ_2$. We have
\begin{align*}
\psi(d_3(a,b,c)) &= \psi((a \op b,c) - (a,c) - (a \op c, b \op c) + (a,b))\\
&=\overline{b \op c}+\overline{b}+\varepsilon(a \op b,c) +\varepsilon(a,c) +\varepsilon(a \op c, b) + \varepsilon(a,b).
\end{align*}
The part $\overline{b \op c}+\overline{b}$ vanishes unless $\{i,j\}=\{1,2\}$, where $b \in O_i$, $c \in O_j$; the part $\varepsilon(a \op b,c) +\varepsilon(a,c)$ vanishes unless $c \in O_1$ and $b \in O_1 \cup O_2$; similarly, $\varepsilon(a \op c, b) + \varepsilon(a,b)$ vanishes unless $b \in O_1$ and $c \in O_1 \cup O_2$. The sum of these three parts is zero in any case. Further, as announced,
\[\theta((z_1,x_1)-(z_1,x_0))=1-0=1.\qedhere\] 
\end{proof}

We finish with another generalisation of the family $U_{m,n}$, borrowed from \cite{HomGraphic}. Given positive integers $n_1,\ldots,n_r$, put $O_i:=\ZZ_{n_i}$, $U_{n_1,\ldots, n_r}:=\sqcup_{i=1}^r O_i$, and 
\[a \op b =\begin{cases} a & \text{ if $a$ and $b$ are from the same } O_i,\\
a+1  & \text{ if $a$ and $b$ are from different } O_i. \end{cases}\]
In terms of the permutations $f_{i,j}$ from \eqref{E:fij}, for each $i$ we impose all the $f_{i,j}$ to be the same cycle $a \mapsto a+1$ on $O_i$. This is an $r$-orbit quandle, with orbits $O_i$. The case $r=2$ covers the family $U_{m,n}$. In the case $r=3$, the parameter matrix is
\[\begin{psmallmatrix}
n_1 & \cdot & \cdot\\
-1 & 1 & \cdot\\ 
\cdot & \cdot & n_2\\
-1 & \cdot & -1\\
\cdot & -n_3 & \cdot\\
\cdot & 1 & -1\\
\end{psmallmatrix}.\]

\begin{pro}\label{P:graphic}
For $r \geq 3$, the parameter group of the quandle $U_{n_1,\ldots, n_r}$ is 
\[G'(U_{n_1,\ldots, n_r}) \cong \ZZ_{\gcd(n_1,\ldots,n_r,2)}.\] 
For $r=3$, its 2nd homology group is
\[H_2(U_{n_1,n_2, n_3}) \cong \ZZ^{9} \oplus \ZZ^r_{\gcd(n_1,n_2,n_3,2)}.\]
In particular, all torsion groups coincide with the whole parameter group:
\[T_1 \cong T_2 \cong T_3 \cong G'.\]
For $r \geq 4$, there is no torsion:
\[H_2(U_{n_1,\ldots, n_r}) \cong \ZZ^{r^2}.\]
\end{pro}

Homology computations for these quandles were done in~\cite{HomGraphic}. Here we correct their result for the $r \geq 4$ case.

\begin{proof}
Recall the presentation~\eqref{E:G'} for the parameter group $G'$. In our case the relations coming from a component $G(M^{(i)})$ can be interpreted as follows:
\begin{enumerate}
\item the generators $x^{(i)}_{j}$ depend on~$i$ only, and can thus be denoted by $g_i$;
\item $g_i^{n_i}=1$ for all $i$.
\end{enumerate}
The inter-component relations, $x^{(i)}_{j-i}x^{(j)}_{i-j}=1$ for $i < j$, become $g_ig_j=1$. Since $r \geq 3$, for any $i,j$ there is a $k \notin \{i,j\}$, and $g_ig_k=g_jg_k=1$ yields $g_i=g_j$. So, one has only one generator $g:=g_i$ (for any $i$). In terms of this generator, the relations become $g^{n_i}=1$ for all $i$, and $g^2=1$. Summarising, one gets a cyclic group of order $\gcd(n_1,\ldots,n_r,2)$.

Now, from Theorem~\ref{T:HomAb} we know that each $T_i$ is a quotient of $G'\cong \ZZ_{\gcd(n_1,\ldots,n_r,2)}$. In the case $r=3$, to see that, say,  $T_1$ is the whole $\ZZ_2$ when all the $n_i$ are even, one can repeat the argument from the proof of Proposition~\ref{P:star}, putting 
\[\overline{a}=a \mod 2 \qquad \text{ for all } a,\]
\begin{align*}
\varepsilon(a,b) &= \begin{cases} 1 &\text{ if } \overline{a} = 1, b \in O_1 \sqcup O_3,\\
0 &\text{ otherwise}.
\end{cases}
\end{align*}

In the case  $r \geq 4$, it remains to show the triviality of, say, $T_1$. Computations below will be done in~$Q_2$, and will be valid for all $a,a' \in O_1$, $b,b'\in O_i$, $c \in O_j$, with $1 \neq i \neq j \neq 1$. Relation $d_3(a,a',b)=0$ in $Q_2$ yields
\[(a+1,a'+1)=(a,a').\]
 Relation $d_3(a,b,a')=0$ yields
\[(a,b+1)-(a,b)=(a+1,a')-(a,a').\]
This expression is independent of~$b$ nor~$a'$; let us denote it by $\varphi(a)$. Finally, relation $d_3(a,b,c)=0$ yields
\begin{equation}\label{E:dabc}
(a+1,b+1)-(a,b)=(a+1,c)-(a,c).
\end{equation}
Given a $d \in O_k$, $k \neq 1$, one can always find $j \notin \{1,i,k\}$ (recall that we have $r \geq 4$ orbits), so  
\[(a+1,b+1)-(a,b)=(a+1,c)-(a,c)=(a+1,d+1)-(a,d).\]
Thus the LHS of~\eqref{E:dabc} is independent of~$b$ (as long as it does not lie in~$O_1$). Let us denote it by $\psi(a)$. Looking at the RHS of~\eqref{E:dabc}, one gets
\[\psi(a)=(a+1,b+1)-(a,b)=(a+1,b+1)-(a,b+1)+(a,b+1)-(a,b)=\psi(a)+\varphi(a),\]
hence $\varphi(a)=0$ for all $a$. But this means that the generator $(a,b \op c)-(a,b)=(a,b+1)-(a,b)$ of~$H_2$ is trivial.
\end{proof}


\bibliographystyle{alpha}
\bibliography{refs}
\end{document}